\documentclass[12pt]{amsart}
\usepackage[pagewise]{lineno}

\usepackage[english]{babel}

\usepackage[letterpaper,top=2cm,bottom=2cm,left=3cm,right=3cm,marginparwidth=1.75cm]{geometry}

\usepackage{amsmath,amssymb,amsfonts,amsthm,dsfont,comment}
\usepackage{graphicx}
\usepackage[colorlinks=true, allcolors=blue]{hyperref}

\newtheorem{theorem}{Theorem}[section]
\newtheorem{proposition}[theorem]{Proposition}
\newtheorem{lemma}[theorem]{Lemma}
\newtheorem{corollary}[theorem]{Corollary}
\newtheorem{problem}[theorem]{Problem}
\newtheorem{prob*}{Problem}
\newtheorem{question}[theorem]{Question}

\newtheorem*{theorem*}{Theorem}

\theoremstyle{definition}

\newtheorem{example}[theorem]{Example}
\newtheorem{remark}[theorem]{Remark}
\title[Finite subgroups of $\PGL_2(K)$ from configurations of skew lines]{Finite subgroups of $\mathrm{PGL}_2(K)$ arising from configurations of skew lines in $\mathbb P^3_K$}
\author{Giuseppe Favacchio}
\address[G. Favacchio]{Dipartimento di Ingegneria, Universita degli studi di Palermo, Viale delle Scienze,
90128 Palermo, Italy}
\email{giuseppe.favacchio@unipa.it}
\date{}
\usepackage{nicematrix,makecell}

\newcommand{\GL}{\mathrm{GL}}  
\newcommand{\PGL}{\mathrm{PGL}}  
\newcommand{\PSL}{\mathrm{PSL}}  
  
\newcommand{\Aut}{\mathrm{Aut}}

\DeclareMathOperator{\Span}{Span}

\DeclareMathOperator{\lcm}{lcm}

\newcommand{\PP}{ \ensuremath{\mathbb{P}}}

\keywords{Skew lines, half-grid configurations, collinearly complete sets,  finite subgroups of PGL2, group actions on projective space}
\subjclass[2020]{14N20, 14L30, 20H20}

\begin{document}
\begin{abstract}
We study finite groups arising from configurations of pairwise skew lines in $\mathbb{P}^3_K$. To such a configuration $\mathcal{L}$ one associates a group $G_{\mathcal{L}}\subset \mathrm{PGL}_2(K)$ acting on each line, and we investigate which finite subgroups of $\mathrm{PGL}_2(K)$ can occur in this way.

Our main tool is a matrix description of skew lines in $\mathbb{P}^3_K$, which gives explicit generators for $G_{\mathcal{L}}$ in terms of matrices in $\mathrm{GL}_2(K)$. In the abelian case, we prove that the relevant matrices are simultaneously upper triangular and obtain explicit families realizing cyclic groups and elementary abelian $p$-groups.  In the non-abelian case, we show that, in non-modular characteristic, no dihedral group $D_n$ with
$n\ge 3$ can occur, while configurations realizing $A_4$, $S_4$, and $A_5$ are constructed explicitly. 

These results also yield new examples of point sets whose general projection is a complete intersection.
\end{abstract}
\maketitle

\section{Introduction and Background}
Let $K$ be an algebraically closed field and let $\PP^3_K=\PP^3$ denote the three-dimensional projective space over $K$.  We begin by recalling the setup introduced in \cite[Section 2]{politus3}.

Consider three distinct skew lines $L_1,L_2,L_3 \subset \PP^3$. One can naturally associate a linear isomorphism
\[
  f_{123}\colon L_1 \longrightarrow L_2
\]
defined geometrically as follows: for each $p\in L_1$, consider the plane spanned by $p$ and $L_3$; this plane intersects $L_2$ in a unique point, which we denote by $q=f_{123}(p)$.

This construction extends in a straightforward way to any finite collection of skew lines. Let ${\mathcal L}=\{L_1,\ldots,L_s\}\subseteq \PP^3$ be a set of $s\geq 3$ distinct skew lines.
Each ordered triple $(L_{i_1},L_{i_2},L_{i_3})$ of lines from $\mathcal L$ 
determines a map
\[
  f_{i_1 i_2 i_3}\colon L_{i_1} \longrightarrow L_{i_2}
\]
defined exactly as above.

Compositions of these maps can also be considered whenever the target of one map coincides with the domain of the next: for instance,
  $f_{\ell m n} \circ f_{i j k}$
is defined whenever $j = \ell$. In general, however, such compositions need not themselves arise from another triple of lines in $\mathcal L$, that is, $f_{j m n}\circ f_{i j k}\colon L_i\to L_m$ need not be of the form $f_{i m  r}$ for any $L_r\in {\mathcal L}$.

In \cite[Definition 2.1.2]{politus3} a groupoid structure on $\mathcal L=\{L_1,\ldots, L_s\}$ is introduced, denoted $C_{\mathcal L}$. The objects of $C_{\mathcal L}$ are the lines $L_1,\ldots,L_s$, and the arrows are all the maps $f_{ijk}$ together with all their possible compositions.

For each $L_i$, the set of endomorphisms
\[
  G_i({\mathcal L}) = {\rm Hom}_{C_{\mathcal L}}(L_i,L_i)
\]
forms a group under composition, naturally viewed as a subgroup of the projective linear group ${\rm Aut}(L_i)\cong \PGL_2(K)$. Since all such groups $G_i({\mathcal L})$ are (noncanonically) isomorphic for $i=1,\ldots,s$, we write $G_{\mathcal L}$ for this common group and refer to it as the group of the groupoid.

Given a point $p\in L_i\in \mathcal L$ we can apply to it all morphisms in $C_{\mathcal L}$ whose source is $L_i$. The collection of all images of $p$ under such maps is called its \emph{orbit} and denoted by $[p]_{\mathcal L}$. 

A natural question is: how can we recognize sets that are unions of such orbits geometrically, without explicitly referring to the maps $f_{ijk}$? To answer this, \cite[Definition 2.1.10]{politus3} introduces the notion of collinear completeness.

Let $Z$ be a set (nonempty but not necessarily finite) of points contained in a finite union of lines $\bigcup\limits_{L_i\in \mathcal L} L_i$.
We say that $Z$ is \emph{collinearly complete with respect to ${\mathcal L}$} if it satisfies the following property:
whenever a line $T$ meets at least three lines in $\mathcal L$ and $T$ contains at least one point of $Z$, then $Z$ contains all of the intersection points of $T$ with the lines of $\mathcal L$.

Then \cite[Theorem~3.0.3]{politus3} characterizes which finite sets $Z$ have this property in terms of the general projection of $Z$ to a plane. Precisely, let $Z \subset \PP^3$ be a set of points with exactly $a$ points
on each of the $b$ lines of $\mathcal L$. Then $Z$ is collinearly complete
for $\mathcal L$ (equivalently, $Z$ is a union of $C_{\mathcal L}$-orbits)
if and only if, for a general projection $\pi : \PP^3 \to \PP^2$, the image
$\pi(Z)$ is a complete intersection of plane curves of degrees $a$ and $b$
(these sets are in general called $(a,b)$-geproci sets), and the curve of
degree $b$ can be chosen as the union of the images of the lines in $\mathcal L$.

The characterization above naturally leads to the following problem: given a configuration $\mathcal L$ of pairwise skew lines in $\PP^3$, determine when the associated group $G_{\mathcal L}$ is finite, and more generally which finite subgroups of $\PGL_2(K)$ can occur in this way.

This is the main question addressed in the present paper. Our approach is based on a matrix description of skew lines in $\PP^3$, which allows us to express the generators of $G_{\mathcal L}$ explicitly and to translate the problem into a concrete linear-algebraic one. In this way we obtain both structural results and explicit constructions, and we relate them to the geometry of collinearly complete point sets and geproci sets.

Partial results on finite orbits and special configurations can be found in \cite{politus3, ganger2024spreads, kettinger2025}. Since $G_{\mathcal L}$ acts faithfully on each line $L_i\in\mathcal L$, it is naturally a subgroup of $\PGL_2(K)$, so the classification of finite subgroups of $\PGL_2(K)$ provides the natural framework for the problem.

\subsection{Finite subgroups of \texorpdfstring{$\PGL_2(K)$}{PGL2(K)}}

We recall the finite subgroups of $\PGL_2(K)$.

If $p=0$, or more generally if $p\nmid |G|$, then every finite subgroup $G\subset \PGL_2(K)$ is isomorphic to one of
\[
C_n,\qquad D_n,\qquad A_4,\qquad S_4,\qquad A_5.
\]
If $p>0$ and $p\mid |G|$, then in addition one may have finite affine groups of the form $C_p^m\rtimes C_n$, as well as groups of Lie type such as $\PSL_2(\mathbb F_q)$ and $\PGL_2(\mathbb F_q)$.
We refer to \cite{Beauville2010,DolgachevIskovskikh2009,Faber2023,Klein1884,leuschke2012} for background and explicit realizations.

\subsection{Overview of Results}
The main problem considered in this paper is the following realization problem: which finite subgroups of $\PGL_2(K)$ can occur as $G_{\mathcal L}$ for a configuration $\mathcal L$ of pairwise skew lines in $\PP^3_K$?

Our main tool is a matrix description of skew lines in $\PP^3_K$. Fixing two reference skew lines, every other line in the configuration can be represented by a matrix $M_i\in\GL_2(K)$, and the generators of $G_{\mathcal L}$ can be written explicitly in terms of these matrices and their differences. This turns the problem into a concrete linear-algebraic one and makes it possible to treat the abelian and non-abelian cases in a uniform way.

In Section~\ref{s.lines<->matrices}, we translate the geometric properties of a configuration of skew lines in $\PP^3$ into linear-algebraic conditions on the associated matrices $M_i\in\GL_2(K)$, and we express the maps $f_{ijk}$ and the generators of $G_{\mathcal L}$ in terms of these matrices (see Lemma~\ref{l.describe f_ijk} and Theorem~\ref{t.describe G_L}).

In the abelian case, treated in Section~\ref{s.abelian}, we prove that after a change of basis the matrices $M_i$ are simultaneously upper triangular, and we describe the corresponding configurations explicitly (Lemma~\ref{l.commute} and Corollary~\ref{c.G abelian}). In particular, if all $M_i$ are diagonal then the lines admit two common transversals and $G_{\mathcal L}$ is cyclic (Proposition~\ref{p.4 lines case 1}), while in positive characteristic the non-semisimple case produces the elementary abelian group $C_p\times C_p$ (Proposition~\ref{p.4 lines case 2}). We conclude this section by solving the realization problem in the case $G_{\mathcal L}\cong C_3$, that is, by completely characterizing the configurations of skew lines $\mathcal L$ such that
$G_{\mathcal L}\cong C_3,$ see Proposition~\ref{p.3 elements} and Remark~\ref{r.3 elements}.

In the non-abelian case, considered in Section \ref{s.non-abelian}, we prove that no dihedral group $D_n$ with $n\ge 3$ and
$\operatorname{char}(K)\nmid n$ can occur as $G_{\mathcal L}$ (Theorem \ref{t.no Dn}). On the other hand, we construct explicit configurations whose associated group is isomorphic to each of the polyhedral groups $A_4$, $S_4$, and $A_5$ (Examples \ref{ex.A5}, \ref{ex.S4}, and \ref{ex.A4}), and we describe the corresponding orbit structure on a distinguished line. We also give an explicit positive-characteristic example realizing a non-abelian affine $p$-semi-elementary group of the form $C_p\rtimes C_{p-1}$
(Example \ref{non-abelian affine}).

Finally, via the characterization of collinearly complete sets from
\cite[Theorem 3.0.3]{politus3}, our results produce new families of half-grid
geproci sets. In the abelian case, they recover and extend the standard
constructions with two transversals in
\cite[Section~4]{politus1}. In the non-abelian case, they yield half-grid
examples with no transversals, governed instead by the actions of the
polyhedral groups $A_4$, $S_4$, and $A_5$. Thus the group-theoretic description
of $G_{\mathcal L}$ reflects a genuine geometric distinction among collinearly
complete point sets arising from skew-line configurations. For recent
developments on geproci sets and related questions we refer to
\cite{politus1, chiantini2021, FM2024, ganger2024spreads, kettinger2023, kettinger2024}.
\ \\
\ \\
\noindent{\bf Acknowledgment.} Throughout the paper, the calculations were performed using the software Macaulay2~\cite{macaulay2}, and we provide the code in the appendix.
The author used a large language model (ChatGPT, OpenAI) for expository feedback during the revision of the manuscript. The author takes full responsibility for the final text and all results, proofs, citations, and computational verification.

The author is a member of GNSAGA-INdAM. The work of the author was supported by the funding PREMIO\_SINGOLI\_RIC\_[2025] from the Department of Engineering, University of Palermo.
\section{Identifying sets of skew lines in \texorpdfstring{$\PP^3_K$}{P3(K)} via matrices in \texorpdfstring{$\GL_2(K)$}{GL2(K)}}\label{s.lines<->matrices}
The main goal of this section is to translate the geometry of skew lines in $\PP^3$ into an explicit linear-algebraic framework. Once two reference skew lines are fixed, every other line in the configuration can be represented by a matrix in $\GL_2(K)$, and the geometric condition that two lines be skew becomes the invertibility of the difference of the corresponding matrices. This point of view is the key tool used throughout the paper: it allows us to describe the maps $f_{ijk}$ explicitly, to obtain concrete generators for
$G_{\mathcal L}$, and ultimately to turn the realization problem for finite groups arising from skew-line configurations into a tractable matrix problem.

Let $K$ be a field.  We regard $\mathbb P^3$ as the projective space
of $K^4$, whose points are the $1$-dimensional subspaces of $K^4$ and
whose lines correspond to $2$-dimensional subspaces of $K^4$.
Consider two pairwise skew lines $L_\infty, L_0\subset \mathbb
P^3$.  Up to applying a projective transformation of $\mathbb P^3$, we
may assume that these lines are given by
\[
L_{\infty}:\; x=y=0,
    \ \ 
L_{0}:\; z=w=0.
\]
Let $V_0,V_\infty\subset K^4$ be the $2$--dimensional vector subspaces
whose projectivizations are $L_0$ and $L_\infty$ respectively.  Then
\[
V_0 = K^2 \oplus \{0\},
\qquad 
V_\infty = \{0\}\oplus K^2,
\]
so that $K^4 = V_0 \oplus V_\infty,$ and
$\mathbb P^3\cong \mathbb P(V_0 \oplus V_\infty).$

A line $L_i$ is skew to $L_0$ and $L_\infty$ if and only if the corresponding $2$-dimensional subspace $V_i$ intersects neither $V_0$ nor $V_{\infty}$  nontrivially. 
In other words, $V_i$ is the graph of an isomorphism $\psi_i:V_0\to V_\infty$ 
$$V_i=\{(v, \psi_i(v))\ | \ v\in K^2 \}.$$
Using the identifications $V_0\cong K^2$ and $V_\infty\cong K^2$, the
map $\psi_i$ is uniquely represented by a nonsingular  matrix $M_i\in \mathrm{GL}_2(K)$.  Thus,
\[
L_i = \{(v,M_iv)\ | \ v\in K^2\}
      \subset \mathbb P(K^2\oplus K^2).
\]
The bijection is explicitly given by
\begin{align*}
 V_i= \langle (e_1,M_ie_1), (e_2,M_ie_2) \rangle
&\ \leftrightarrow L_i= \mathbb{P}(V_i) = \langle [1:0:a:c], [0:1:b:d] \rangle \\
   M_i=\left(\begin{array}{cc}
   a  & b \\
   c  & d
\end{array}\right) &\ \leftrightarrow L_i: \begin{cases}
    z=ax+by\\
    w=cx+dy
\end{cases}
\end{align*}

Now fix $L_1$ to be $z-x=w-y=0$; then $\psi_1$ is the identity map $id_{K^{2}}$ and $$M_1= \left(\begin{array}{cc}
   1  & 0 \\
   0  & 1
\end{array}\right).$$
\begin{remark}
    A line $L_i$ is skew to $L_1$ if and only if $M_i$ has eigenvalues different from 1. Indeed, $L_i\cap L_1\neq\emptyset$ if and
only if $(v,M_iv)=(v,v)$ for some $v\neq 0$, i.e., $M_iv=v$.
\end{remark}

\begin{lemma}\label{l.skew condition}
    The lines $L_i$ and $L_j$ are skew if and only if $M_i-M_j$ is nonsingular.
\end{lemma}
\begin{proof}
The intersection $L_i\cap L_j$ is nonempty if and only if  
$$L_i\cap L_j=\{(v,M_iv)\ |\ v \in K^2\}\cap \{(w,M_jw)\ |\ w \in K^2\}$$ contains a nonzero vector. That is, $(v,M_iv)=(w,M_jw)$ for some $v,w\in K^2\setminus\{0\}$. This is equivalent to  $v=w$ and  $M_iv=M_jv$, i.e.,  $(M_i-M_j)(v)=0$ that is $M_i-M_j$ is singular.     
\end{proof}
By construction, the matrices $M_i$ are well defined only for lines $L_i$ skew to both $L_0$ and $L_\infty$.
For notational convenience we also write  $$M_0= \begin{pmatrix}
  0  & 0 \\
   0  & 0
\end{pmatrix}\ \text{and}\ M_{\infty}= \begin{pmatrix}
   \infty  & 0 \\
   0  & \infty
\end{pmatrix}$$
to indicate that $L_0$ and $L_\infty$ play the role of ``lines at $0$ and $\infty$''; these symbols are not elements of $\GL_2(K)$, but they fit our notation for the formulas below.

\begin{lemma}\label{l.transv 5 lines} 
    Let $\mathcal L=\{L_0,  L_{\infty}, L_1, L_2, L_3\}$ be a set of five skew lines in $\mathbb P^3$ and let  $M_0$,  $M_{\infty}$, $M_1$, $M_2$, $M_3$ the corresponding matrices. Then  the following are equivalent:
    \begin{enumerate}
        \item there exists a line transversal  to $\mathcal L$;
        \item $M_2$ and $M_3$ have a common eigenvector;
        \item the matrices $M_2$ and $M_3$ are simultaneously triangularizable;
        \item $\det([M_2, M_3])=\det(M_2M_3-M_3M_2)=0.$
    \end{enumerate} 
\end{lemma}
\begin{proof}
We first prove the equivalence of $(1)$ and $(2)$.

\smallskip
\noindent
$(1)\iff (2)$.
    A line $T$ meets $L_0$ and $L_{\infty}$ if and only if $T$ is the span of $(v,0)$ and $(0,w)$ for some $v,w\in K^2$. 
The additional condition that $T$ meets $L_1=\{(u,u)\ |\ u\in K^2\}$ means
that there exists $u\neq 0$ with $(u,u)\in T$, i.e.,
\[
(u,u)=\alpha (v,0)+\beta(0,w)=(\alpha v,\beta w)
\]
for some $\alpha,\beta\in K$. Hence $u=\alpha v=\beta w$, so $w$ is proportional
to $v$. After rescaling we may assume $w=v$, and therefore $T$  is the span of  $(v,0)$ and $(0,v).$

    Moreover, $T$ meets $L_2$ if and only if $(u,M_2u)\in T$ for some $u\in K^2,$ i.e., $$(u,M_2u)=\alpha(v,0)+\beta (0,v)=(\alpha v, \beta v), \ \ \alpha,\beta\in K.$$
That is, $u=\alpha v$ and $M_2\alpha v=\beta v$, which means that $v$ is an eigenvector for $M_2.$
The same argument shows that $v$ is an eigenvector for $M_3.$

\smallskip
\noindent
$(2)\iff (3)$.A vector $v$ is a common eigenvector for $M_2$ and $M_3$ if and only if in a basis $[v,w]$ of $K^2$, the matrices  $M_2$ and $M_3$ are both upper triangular (and thus simultaneous triangularizable).

\smallskip
\noindent
$(2)\iff (4)$.   Let $C=[M_2,M_3]=M_2M_3-M_3M_2$.

If $v$ is a common eigenvector for $M_2$ and $M_3$ with eigenvalues $\lambda_2$ and $\lambda_3$ respectively, then
\[
(M_2M_3-M_3M_2)(v) = \lambda_3 M_2v-\lambda_2 M_3v = \lambda_3\lambda_2v-\lambda_2\lambda_3v = 0,
\]
so $C$ is singular and hence $\det(C)=0$.

Conversely, it is a general fact that $\mathrm{tr}([A,B])=0$ for any matrices $A,B$. 
Hence, if $\det(C)=0$ then the characteristic polynomial of $C$ is $$\chi_C(\lambda)= \lambda^2 - \mathrm{tr}(C)\lambda + \det(C)=\lambda^2.$$ 
Thus $C$ is nilpotent, in particular $C^2=0$.

If $C=0$ then $M_2$ and $M_3$ commute and, since $K$ is algebraically closed and $M_2,M_3\in \GL_2(K)$, they are simultaneously triangularizable; in particular they have a common eigenvector.

Assume now $C\neq 0.$  Since $\det(P^{-1}CP)=\det(C)$ for every invertible
$P$, we are free to choose a convenient basis of $K^2$ in which to work.

\smallskip

\emph{Case (a): $M_2$ is diagonalizable.}
Then, up to a change of basis, we may write
\[
M_2 = \begin{pmatrix} a_{11} & 0 \\[2pt] 0 & a_{22} \end{pmatrix},\qquad
M_3 = \begin{pmatrix} b_{11} & b_{12} \\[2pt] b_{21} & b_{22} \end{pmatrix}.
\]
A direct computation gives
\[
C
=
\begin{pmatrix}
 0 & (a_{11}-a_{22}) b_{12} \\
 (a_{22}-a_{11}) b_{21} & 0
\end{pmatrix}\implies 
\det(C)
= b_{12} b_{21} (a_{11}-a_{22})^2.
\]
If either $b_{21}=0$ or $b_{12}=0$ then $M_2$ and $M_3$ share a common eigenvector. If $a_{11}=a_{22}$, then $M_2=a_{11} I$ is a scalar matrix.
Since $K$ is algebraically closed, $M_3$ has at least one eigenvector $v$,
and then $M_2 v = a_{11} v$ as well, so $v$ is a common eigenvector.
  
  \smallskip
\emph{Case (b): $M_2$ is not diagonalizable.}
  Since $K$ is algebraically closed, $M_2$ has an eigenvalue $a_{11}$ and a Jordan form.
Up to a change of basis, we may assume
\[
M_2 = \begin{pmatrix} a_{11} & a_{12} \\[2pt] 0 & a_{11} \end{pmatrix},
\ a_{12}\neq 0,
\ \ 
\text{and set}\ \ 
M_3 = \begin{pmatrix} b_{11} & b_{12} \\[2pt] b_{21} & b_{22} \end{pmatrix}.
\]
A direct computation yields
\[
C
=
\begin{pmatrix}
 a_{12} b_{21} & -a_{12} b_{11} + a_{12} b_{22} \\
 0             & -a_{12} b_{21}
\end{pmatrix}
\implies 
\det(C) = a_{12}^2 b_{21}^2.
\]
Since $a_{12}\neq 0$ and $\det(C)=0$, we must have $b_{21}=0$, so $M_3$
is upper triangular in the same basis as $M_2$. This implies that $M_2$ and $M_3$ share a common eigenvector.
\end{proof}

From the proof of Lemma \ref{l.transv 5 lines} we obtain an immediate description of transversals. 
\begin{corollary}
    Let $\mathcal L=\{L_0,  L_{\infty}, L_1, L_2, L_3\}$ be a set of five skew lines in $\mathbb P^3$ admitting a  transversal,  and let  $M_0$,  $M_{\infty}$, $M_1$, $M_2$, $M_3$ be the corresponding matrices. Then, a line $T$ is transversal  to $\mathcal L$ if and only if $$T=\langle(v,0),(0,v)\rangle$$ for some $v\in K^2$ that is an eigenvector for both $M_2$ and $M_3$.  
\end{corollary}

In the next lemma we write the maps $f_{ijk}$ in terms of the matrices associated to the lines.

\begin{lemma}\label{l.describe f_ijk}
Let $\mathcal L=\{L_0,L_{\infty},L_1,\ldots,L_r\}$ be a set of skew lines in
$\mathbb P^3$, and let $M_0=0$, $M_1=I$, $M_2,\ldots,M_r$ be the
corresponding matrices. For any triple of distinct indices, the map
$f_{ijk}:L_i\to L_j$ is represented, under the parametrizations
$L_i\simeq \mathbb P(K^2)$, by $F_{ijk}\in\PGL_2(K)$ as follows.
\begin{enumerate}
\item If $i,j,k\in\{0,1,\ldots,r\}$, then
\[
F_{ijk}=[(M_j-M_k)^{-1}(M_i-M_k)].
\]
\item If one of the three indices is $\infty$, then
\[
F_{ij\infty}=[I],\qquad
F_{i\infty k}=[M_i-M_k],\qquad
F_{\infty jk}=[(M_j-M_k)^{-1}].
\]
\end{enumerate}
\end{lemma}
\begin{proof}
\noindent (1) Fix a point $p_i(v)=(v,M_iv)\in L_i$ for some nonzero $v\in K^2$. The plane
spanned by $p_i(v)$ and $L_k$ is
\[
\Pi_k(v)=\{\lambda(v,M_iv)+(w,M_kw)\mid \lambda\in K,\ w\in K^2\}.
\]
Since $L_j$ is disjoint from $L_k$, the point
$L_j\cap \Pi_k(v)$ does not lie on $L_k$. Hence $\lambda\neq 0$, and
after rescaling we may take $\lambda=1$.
A point $p_j(u)=(u,M_ju)\in L_j\cap\Pi_k(v)$ satisfies
\[
(u,M_ju)=(v,M_iv)+(w,M_kw)
\]
for some $w\in K^2$, which gives
\[
\begin{cases}
u=v+w,\\
M_ju=M_iv+M_kw.
\end{cases}
\]
Substituting $w=u-v$ into the second equation:
\[
M_ju=M_iv+M_k(u-v)
\implies
(M_j-M_k)u=(M_i-M_k)v.
\]
Since $L_j$ and $L_k$ are skew, Lemma~\ref{l.skew condition} implies that 
$M_j-M_k$ is nonsingular, so
\[
u=(M_j-M_k)^{-1}(M_i-M_k)v,
\]
and therefore $F_{ijk}=[(M_j-M_k)^{-1}(M_i-M_k)]$.
\medskip

\noindent(2) Recall that $L_0=\{(v,0)\mid v\in K^2\}$ and 
$L_\infty=\{(0,u)\mid u\in K^2\}$.

\begin{itemize}
    \item[$F_{ij\infty}$:] The plane spanned by $p_i(v)=(v,M_iv)$ 
and $L_\infty=\{(0,u)\mid u\in K^2\}$ is
\[
\Pi_\infty(v)=\{(v,M_iv)+(0,u)\mid u\in K^2\}=\{(v,M_iv+u)\mid u\in K^2\}.
\]
This plane meets $L_j=\{(w,M_jw)\mid w\in K^2\}$ when $(w,M_jw)=(v,M_iv+u)$, 
giving $w=v$ and $u=M_jv-M_iv$. Thus the image of $p_i(v)$ in $L_j$ is 
$p_j(v)=(v,M_jv)$, which shows that $F_{ij\infty}=[I]$.\medskip
\item[$F_{i\infty k}$:] The plane spanned by $p_i(v)=(v,M_iv)$ 
and $L_k=\{(w,M_kw)\mid w\in K^2\}$ meets $L_\infty=\{(0,u)\mid u\in K^2\}$ 
when
\[
(0,u)=(v,M_iv)+(w,M_kw)
\]
for some $w\in K^2$, giving $w=-v$ and $u=M_iv-M_kv=(M_i-M_k)v$. Thus the 
image of $p_i(v)$ in $L_\infty$ is $(0,(M_i-M_k)v)$, which shows that 
$F_{i\infty k}=[M_i-M_k]$.\medskip
\item[$F_{\infty jk}$:] A point of $L_\infty$ has the form 
$(0,u)$. The plane spanned by $(0,u)$ and $L_k=\{(w,M_kw)\mid w\in K^2\}$ 
meets $L_j=\{(v,M_jv)\mid v\in K^2\}$ when
\[
(v,M_jv)=(0,u)+(w,M_kw)
\]
for some $w\in K^2$, giving $v=w$ and $M_jv=u+M_kv$, so 
$u=(M_j-M_k)v$. Since $M_j-M_k$ is nonsingular, we get 
$v=(M_j-M_k)^{-1}u$, which shows that $F_{\infty jk}=[(M_j-M_k)^{-1}]$. \qedhere
\end{itemize} 
\end{proof}

The preceding lemma describes the individual arrows of the groupoid. To obtain
the group $G_{\mathcal L}$, one has to consider closed compositions, i.e.,
compositions which start and end on the same line. The following theorem is the
matrix version of \cite[Proposition~2.1.3]{politus3}.
\begin{theorem}\label{t.describe G_L}
Let
\[
\mathcal L=\{L_0,L_{\infty},L_1,\ldots,L_r\}
\]
be a set of skew lines in $\mathbb P^3$, and let
$M_0=0$, $M_1=I$, $M_2,\ldots,M_r$ be the corresponding matrices.
Let $\mathcal L'\subseteq \mathcal L$ be any subconfiguration with at least
three lines. Then the group $G_{\mathcal L'}$ is generated by the closed
compositions represented by
\[
F_{jih}F_{ijk}
\qquad\text{and}\qquad
F_{kic}F_{jkb}F_{ija},
\]
where all indices are chosen among the lines of $\mathcal L'$, and each
triple appearing in an $F_{\alpha\beta\gamma}$ consists of distinct indices.
\end{theorem}
\begin{proof}
By \cite[Proposition~2.1.3]{politus3}, the group $G_{\mathcal L'}$ is generated
by closed compositions of the maps $f_{ijk}$ of the two forms above. By
Lemma~\ref{l.describe f_ijk}, each map $f_{ijk}$ is represented, under the
parametrizations $L_i\simeq \mathbb P(K^2)$, by the corresponding element
$F_{ijk}\in\PGL_2(K)$. Hence the closed compositions are represented by the
products displayed above.
\end{proof}
When the configuration contains the three distinguished lines
$L_0,L_\infty,L_1$, the general closed-composition generators simplify
substantially. In this normalization they are generated by the difference
classes $[M_i-M_j]$.
\begin{corollary}\label{c.describe G_L}
With notation as above, assume that
$
\mathcal L'=\mathcal L=\{L_0,L_{\infty},L_1,\ldots,L_r\}.
$ 
Then
\[
G_{\mathcal L}
=
\left\langle [M_i-M_j]\ \middle|\ 0\le i\neq j\le r \right\rangle
\subseteq \PGL_2(K).
\]
\end{corollary}
\begin{proof}
Let
\[
H=\left\langle [M_i-M_j]\mid 0\le i\neq j\le r\right\rangle
\subseteq \PGL_2(K).
\]

We first show that $G_{\mathcal L}\subseteq H$. By
Theorem~\ref{t.describe G_L}, the group $G_{\mathcal L}$ is generated by
closed compositions represented by
$
F_{jih}F_{ijk}
\ \text{and}\
F_{kic}F_{jkb}F_{ija}.$

By Lemma~\ref{l.describe f_ijk}, each factor $F_{\alpha\beta\gamma}$ is
represented either by the identity, or by a class of the form
$[M_a-M_b]$, or by the inverse of such a class. Hence every such closed
composition belongs to $H$. Therefore $G_{\mathcal L}\subseteq H$.

Conversely, we show that $H\subseteq G_{\mathcal L}$. Since
$L_{\infty},L_0,L_1$ belong to the configuration, Lemma~\ref{l.describe f_ijk}
gives for $i>1$
\[
F_{1\infty i}F_{\infty 1 0}
=
[I-M_i]\cdot[I]
=
[I-M_i]\in G_{\mathcal L}.
\]
Moreover,
\[
F_{i\infty 0}F_{\infty i 1}
=
[M_i]\,[M_i-I]^{-1}
\in G_{\mathcal L}.
\]
Since $[I-M_i]\in G_{\mathcal L}$, it follows that $[M_i]\in G_{\mathcal L}$.
Finally,
\[
F_{i\infty j}F_{\infty i0}
=
[M_i-M_j]\,[M_i]^{-1}
\in G_{\mathcal L}.
\]
Thus $[M_i-M_j]\in G_{\mathcal L}$ for all $0\le i\neq j\le r$, and hence
$H\subseteq G_{\mathcal L}$. 
Therefore $G_{\mathcal L}=H$.
\end{proof}

\begin{corollary}    Let $\mathcal L=\{L_0,  L_{\infty}, L_1, \ldots, L_r\}$ be a set of skew lines in $\mathbb P^3$ and let  $M_0$,  $M_1$, $\ldots$, $M_r$ the corresponding matrices. 
If $G_{\mathcal L}$ is finite then  for every choice of the distinct indices $i,j\in\{0,\ldots, r\}$  the matrix $M_i-M_j$ has the property that the ratio of its two eigenvalues is a root of unity. 
\end{corollary}
\begin{proof}
By Corollary~\ref{c.describe G_L}, the class $[M_i-M_j]$ belongs to
$G_{\mathcal L}$ for every $0\le i\neq j\le r$. If $G_{\mathcal L}$ is
finite, then each such class has finite order in $\PGL_2(K)$. This is
equivalent to the ratio of the two eigenvalues of $M_i-M_j$ being a root of
unity.
\end{proof}

Now fix a line $L\in \mathcal L$.  The group $G_{\mathcal L}$ acts on $L$ by automorphisms via its identification with a subgroup of $\Aut(L)\cong \PGL_2(K)$; this action is faithful.

\begin{corollary}  Let $\mathcal L=\{L_0,  L_{\infty}, L_1, \ldots, L_r\}$ be a set of skew lines in $\mathbb P^3$. If $G_{\mathcal L}$ is finite and $L\in\mathcal L$ then 

\begin{enumerate}
    \item for any point $p=(v,Mv) \in L$,
$$|G_{\mathcal L}|=|[p]_{\mathcal L}\cap L|\cdot |(G_{\mathcal L})_p|$$
where $(G_{\mathcal L})_p=\{g\in G_{\mathcal L}\ |\ g(p)=p \}$ is the stabilizer of $p$ in $G_{\mathcal L}$. 
\item For a general point $p\in L$, we have $(G_{\mathcal L})_p = \{1\}$, hence
$|[p]_{\mathcal L}\cap L| = |G|$.

\item For a point $p=(v,Mv) \in L$, we have $|[p]_{\mathcal L}\cap L| < |G_{\mathcal L}|$
if and only if $p$ is fixed by some nontrivial element of $G_{\mathcal L}$.
Equivalently, after choosing lifts of $G_{\mathcal L}$ to $\GL_2(K)$, the vector $v$ is an eigenvector of a nontrivial matrix representing an element
of $G_{\mathcal L}$.
\end{enumerate}
\end{corollary}
\begin{proof}
(1) By definition, $[p]_{\mathcal L}\cap L_i$ is the orbit of $p$ under $G_{\mathcal L}$ acting on $L_i$, hence the equality is the orbit--stabilizer theorem.

(2) For a general point $p\in L_i$, no nontrivial projective linear transformation in $G_{\mathcal L}$ fixes $p$; thus $(G_{\mathcal L})_p=\{1\}$ and the statement follows from (1).

(3) A point $p=(v,M_iv)$ is fixed by a nontrivial element $g\in G_{\mathcal L}$ if and only if the corresponding matrix in $\GL_2(K)$ has $v$ as an eigenvector. Again using (1), we have $|[p]_{\mathcal L}\cap L_i|<|G_{\mathcal L}|$ if and only if the stabilizer $(G_{\mathcal L})_p$ is nontrivial, which is equivalent to the existence of such an element.
\end{proof}

In the next two sections we analyze the possible groups $G_{\mathcal L} \subset \PGL_2(K)$,
taking into account the characteristic of $K$. We split the discussion into two
cases, according to whether $G_{\mathcal L}$ is abelian or non-abelian.

\section{The abelian case}\label{s.abelian}
In this section we study the case where $G_{\mathcal L}$ is abelian. Using the matrix description from Section \ref{s.lines<->matrices}, we show that, after a suitable change of basis, all the matrices $M_i$ are simultaneously upper triangular, and we analyze separately the diagonal and non-diagonal cases. This leads to an explicit description of the corresponding configurations of lines and to families realizing cyclic groups and $p$–semi-elementary groups $(C_p)^m\rtimes C_n$.

\begin{remark}
The lines $\mathcal L=\{L_0, L_{\infty}, L_1,\ldots,L_r\}$ are contained in the same smooth quadric surface 
if and only if all the matrices $M_2,\ldots,M_r$
are scalar matrices, i.e., $M_j = a_j I$ for some $a_j\in K$.  
Equivalently, each $M_j$ has a single eigenvalue whose geometric multiplicity is $2$.
In this situation, all lines $L_j$ belong to the same ruling of the smooth quadric defined by $xw-yz=0$,  and every line in the opposite ruling is a transversal to all of them.  Consequently, the group $G_{\mathcal L}$ is abelian (in fact trivial).
\end{remark}

\begin{lemma}\label{l.commute} Let $A,B\in \GL_2(K)$ and consider their class 
     $[A],[B]\in \PGL_2(K).$ Then $[A][B]=[B][A]$ if and only if either
    \begin{enumerate}
        \item $A,B$ are simultaneously diagonalizable;
        \item $A,B$ are both non-diagonalizable and share the same eigenspace; 
        \item either $A$ or $B$ is a multiple of the identity.
    \end{enumerate}
\end{lemma}
\begin{proof}
    If $[A]$ and $[B]$ commute then $AB=\lambda BA$ for some $\lambda\in K^*.$
    
    If $A$ is diagonalizable, in a suitable basis $$A=\begin{pmatrix}
            a_0&0\\
            0&d_0
        \end{pmatrix}\ \ \ \text{and} \ \ B=\begin{pmatrix}
            a_1&b_1\\
            c_1&d_1
        \end{pmatrix}$$
A direct computation shows that
        $$AB=\begin{pmatrix}
            a_0a_1&a_0b_1\\
            d_0c_1&a_0d_1
        \end{pmatrix}, \ \ \ \ BA=\begin{pmatrix}
            a_0a_1&d_0b_1\\
            a_0c_1&d_0d_1
        \end{pmatrix}.$$
    Then
        $[AB]=[BA]$ if and only if either $a_0=d_0$ or $b_1=c_1=0.$

   \smallskip
   \noindent
       If $A$ is not diagonalizable, in a suitable basis $$A=\begin{pmatrix}
            a_0&b_0\\
            0&a_0
        \end{pmatrix}, \ b_0\neq 0 \ \ \text{and} \ \ B=\begin{pmatrix}
            a_1&b_1\\
            c_1&d_1
        \end{pmatrix}$$
       Again, by a direct computation we have
        $$AB=\begin{pmatrix}
           a_0a_1+b_0c_1&a_0b_1+b_0d_1\\
            a_0c_1&a_0d_1
        \end{pmatrix}, \ \ \ \ BA=\begin{pmatrix}
            a_0a_1&b_0a_1+a_0b_1\\
            a_0c_1&c_1b_0+a_0d_1
        \end{pmatrix},$$
       Thus $[AB]=[BA]$ if and only if  $c_1=0$ and $a_1=d_1.$
\end{proof}

In the following result we translate some of the results from \cite[Section~2]{politus3} into the context of this paper. In particular, item (2) assumes a weaker condition than \cite[Proposition 2.1.3]{politus3} since the projective classes $[M_i]$ are not in general generators of $G_{\mathcal L}.$
 \begin{corollary}[Cf. Theorem 2.1.22 in  \cite{politus3}]\label{c.G abelian}   Let $\mathcal L=\{L_0,  L_{\infty}, L_1, \ldots, L_r\}$ be a set of skew lines in $\mathbb P^3$  with corresponding matrices $M_0$,   $M_1$, $\ldots$, $M_r$. Then the following are equivalent
\begin{enumerate}
    \item $G_{\mathcal L}$ is abelian;
    \item the projective classes $[M_i]$ and $[M_j]\in\PGL_2(K)$ commute for all $i,j$; 
    \item either $\mathcal L$ has (at least) two distinct transversals or there is only one transversal to $\mathcal L$ and a partition $$\mathcal  L=\{L_0,L_{\infty}, L_1\}\cup \mathcal L_A\cup \mathcal L_B$$  such that  $\mathcal L_A$ has (at least) two distinct transversals
     and all the lines in $\mathcal L_B$ are tangent to the quadric $xw-yz$ at a point lying on the unique transversal to $\mathcal L$. 
\end{enumerate} 
\end{corollary}
\begin{proof} $(1)\Rightarrow (2)$. If $G_{\mathcal L}$ is abelian then, in particular, the elements $[M_i],[M_j]\in G_{\mathcal L}$ commute for every $i,j$.

$(2)\Rightarrow (1)$. By Corollary~\ref{c.describe G_L}, each generator of $G_{\mathcal L}$ is of
the form $[M_a-M_b]$.
Since all $[M_a]$ commute, the same holds for all $[(M_a-M_b)]$, their inverses, and thus their products. Hence all generators commute and $G_{\mathcal L}$ is abelian.

$(2)\iff (3)$: By Lemma~\ref{l.commute}, commuting projective classes in $\PGL_2(K)$ are either simultaneously diagonalizable, or simultaneously upper triangular with a common eigenspace, or one of them is scalar. Translating this into the geometry of the lines $L_i$ in $\PP^3$ yields exactly the two configurations described in (3):
simultaneous diagonalizability corresponds to the existence of at least two distinct transversals (two distinct eigenspaces), while the case of a common unique eigenspace corresponds to all lines being tangent to the quadric along a single ruling.
\end{proof}

\begin{corollary}[See also Proposition 2.1.20 in \cite{politus3}]
    Let $G_{\mathcal L}$ be abelian and finite. If $char(K)=0$ then there are two distinct transversals to $\mathcal L.$ 
\end{corollary}
\begin{proof}
    This result is an immediate consequence of Corollary \ref{c.G abelian}, since   
    \[\begin{pmatrix}
        a&b\\
        0&a
    \end{pmatrix}^n=\begin{pmatrix}
        a^n&na^{n-1}b\\
        0&a^{n}
    \end{pmatrix}.\] Thus, the finite order of $G_{\mathcal L}$ implies that  $n a^{n-1} b = 0$ for some $n$. In characteristic $0$ this can  only happen if either $a = 0$, which is excluded because the  matrix is non-singular, or $b = 0$, which is excluded because  we are assuming that the matrices are not diagonalizable.
 \end{proof}

The next proposition gives explicit classes of half-grid geproci sets on four lines, obtained as orbits of a general point. For a classification of the
orbits on four skew lines related to half-grid geproci sets with two transversals, see \cite[Section~3.5]{politus3}.

\begin{proposition}[$G_{\mathcal L}\cong C_n$]\label{p.4 lines case 1}
Let $L_2$ be the line corresponding to  $$M_2=\begin{pmatrix}
        a&0\\
        0&d
    \end{pmatrix}\ \ \text{for some }\ a,d\in K\setminus \{0,1\}.$$
Let $\mathcal L=\{L_{\infty}, L_0, L_1, L_2\}$. Then $G_{\mathcal L}$ is abelian and is isomorphic to the subgroup of $K^*$ generated by $\dfrac{a}{d}$ and $\dfrac{a-1}{d-1}$. 
Moreover, $G_{\mathcal L}$ is finite if and only if \[
\begin{cases}
    a=u_1\left(\dfrac{1-u_2}{u_1-u_2}\right)\\
    d=\left(\dfrac{1-u_2}{u_1-u_2}\right)\\
\end{cases}
\]
for some $u_1, u_2\neq 1$ distinct roots of unity. Moreover for $u_1$ of order $m$ and $u_2$ of order~$n$, the group $G_{\mathcal L}$ has order $\lcm(m,n)$. That is $G_{\mathcal L}$ is the cyclic group of order $\lcm(m,n)$.
\end{proposition} 
\begin{proof}
    $G_{\mathcal L}$ is generated by $$[M_2]=\left[\begin{pmatrix}
        a&0\\
        0&d
    \end{pmatrix}\right]=\left[\begin{pmatrix}
        a/d&0\\
        0&1
    \end{pmatrix}\right] $$and$$  [M_2-M_1]=\left[\begin{pmatrix}
        a-1&0\\
        0&d-1
    \end{pmatrix}\right]=\left[\begin{pmatrix}
        \dfrac{(a-1)}{(d-1)}&0\\
        0&1
    \end{pmatrix}\right].$$
    $G_{\mathcal L}$ is finite if and only if the ratio of the eigenvalues of $M_2$ and $M_2-I$ are roots of unity. Thus set $\dfrac{a}{d}=u_1$ and $\dfrac{a-1}{d-1}=u_2$ for some $u_1$ and $u_2$ distinct roots of unity. A straightforward computation gives the conditions in the statement. 
\end{proof}
\begin{remark}
    Since $u_1$ and $u_2$ have to be distinct and different from $1$ $G_{\mathcal L}$ cannot be $C_2$ the cyclic group of order 2. To get $C_3$ let $\varepsilon$ be a primitive third-root of unity and set $u_1=\varepsilon$ and $u_2=\varepsilon^2=-\varepsilon-1$. Then $$\begin{cases}
    a=\varepsilon\left(\dfrac{1-\varepsilon^2}{\varepsilon-\varepsilon^2}\right)=\left(\dfrac{1-\varepsilon^2}{1-\varepsilon}\right)=1+\varepsilon=-\varepsilon^2 \\
    d=\left(\dfrac{1-\varepsilon^2}{\varepsilon-\varepsilon^2}\right)=\left(\dfrac{1-\varepsilon^2}{\varepsilon(1-\varepsilon)}\right)= \dfrac{1+\varepsilon}{\varepsilon}= \dfrac{-\varepsilon^2}{\varepsilon}=-\varepsilon.\\
\end{cases}$$
\end{remark}
In general we have the following.
\begin{theorem}
    Let  $\mathcal L$ be the set of lines corresponding to the following matrices $$\left\{M_{\infty},\,M_0,\,\left(\!\begin{array}{cc}
1&0\\
0&1
\end{array}\!\right),\,\left(\!\begin{array}{cc}
a_2&0\\
0&d_2
\end{array}\!\right),\, \cdots, \,\left(\!\begin{array}{cc}
a_n&0\\
0&d_n
\end{array}\!\right)\right\}.$$
Set $a_0=b_0=0$ and $a_1=b_1=1$. 
Assume $\mathcal L$ is a set of skew lines, i.e., $a_i\neq a_j$ and $d_i\neq d_j$ for each $i\neq j$.  Then  
$G_{\mathcal{L}}$ is abelian and is isomorphic to the subgroup of $K^*$ generated by $\dfrac{a_i-a_j}{d_i-d_j}$ for any $i\neq j$ in $\{0,\ldots, n\}$.  In particular,  it is finite if and only if all the $\dfrac{a_i-a_j}{d_i-d_j}$ are a root of unity.  The order of $G_{\mathcal L}$ is the least common multiple of the order of these roots of unity.
\end{theorem}

\begin{question}Find all possible $a_0, \ldots, a_n$ and  $d_0, \ldots, d_n$ such that   $\dfrac{a_i-a_j}{d_i-d_j}$ is a root of unity for any $i\neq j$.
\end{question}

\begin{proposition}[$G_{\mathcal L}\cong C_p\times C_p$]\label{p.4 lines case 2} Let $K$ be an algebraically closed field of characteristic $p>0$, and let
$\mathbb F_p \subset K$ denote the prime field.
Let $L_2$ be the line corresponding to  $$M_2=\begin{pmatrix}
        a&b\\
        0&a
    \end{pmatrix}\ \ \text{for some }\ a\in K\setminus \mathbb F_p\ \text{and}\ b\neq 0.$$
Let $\mathcal L=\{L_{\infty}, L_0, L_1, L_2\}$. Then $G_{\mathcal L}$ is abelian of order $p^2$ and is isomorphic $\mathbb F_p\times \mathbb F_p \cong C_p\times C_p$.
\end{proposition} 
\begin{proof}By Corollary~\ref{c.describe G_L}, the group $G_{\mathcal L}$ is generated by
$g_1=[M_2]$ and $g_2=[M_2-I]$.
    On the line $\mathbb P^2_K$ choose the affine coordinate $[x:y]\to t=\dfrac{x}{y}$. Then 
    $$M_2\begin{pmatrix}
        x\\
        y
    \end{pmatrix}=\begin{pmatrix}
        a&b\\
        0&a
    \end{pmatrix}\begin{pmatrix}
        x\\
        y
    \end{pmatrix}=\begin{pmatrix}
        ax+by\\
        ay
    \end{pmatrix}\ \text{which gives }\ g_1(t)= t+ \dfrac{b}{a}.$$
    Analogously $g_2(t)=t+\dfrac{b}{a-1}.$ In characteristic $p$ we have $g_1^p(t)=t+p \dfrac{b}{a}=t$ and similarly for $g_2.$ 
    Translations commute, so the subgroup
generated by $g_1$ and $g_2$ is abelian of exponent $p$.

It remains to show that $g_1$ and $g_2$ are $\mathbb F_p$-linearly independent as
translations. This is equivalent to show that $\left(\dfrac{b}{a}\right)\left(\dfrac{a-1}{b}\right)=
\dfrac{a}{a-1}\notin \mathbb F_p. $
    Suppose $\dfrac{a}{a-1}\in\mathbb F_p$. Then there exists $c\in\mathbb F_p$ such that
\[
\frac{a}{a-1}=c
\quad\Longrightarrow\quad
a=c(a-1)=ca-c
\quad\Longrightarrow\quad
(1-c)a=-c
\quad\Longrightarrow\quad
a=\frac{-c}{1-c}\in\mathbb F_p,
\]
a contradiction to $a\notin\mathbb F_p$.
    Thus $|G_{\mathcal L}|=p^2$ and
it is abstractly
    \[G_{\mathcal L}=\langle g_1,g_2\ |\ g_1^p=g_2^p=1, g_1g_2=g_2g_1 \rangle\cong C_p\times C_p.\qedhere\]
\end{proof}

\begin{remark}
Assume $K$ is an algebraically closed field of characteristic $p>0$ and let $\mathcal L$ be the set of lines corresponding to $M_0,M_{\infty}, M_1$ and 
\[
M_i=
\begin{pmatrix}
a_i & b\\
0 & a_i
\end{pmatrix},
\qquad a_i\in K\setminus \mathbb F_p, \ i=2,\ldots, n,\ \text{and}\ b\neq 0.
\]
Then  every
element of $G_{\mathcal L}$ is represented in $\GL_2(K)$ by an upper triangular
 matrix of order $p$:
\[
T_u =
\begin{pmatrix}
1 & u\\
0 & 1
\end{pmatrix},
\qquad u\in K,
\]
which acts on the affine coordinate $t=x/y$ on $\PP^1_K$ by
$
T_u :\ t \longmapsto t+u.$

In particular, $G_{\mathcal L}$ consists entirely of translations, and there is
a finite $\mathbb F_p$-subspace $U\subset K$ such that
\[
G_{\mathcal L}
=
\{\, [T_u] \mid u\in U \,\}
\ \cong\ (C_p)^m,
\quad m=\dim_{\mathbb F_p} U.
\]
Our construction with upper triangular matrices
realizes all such groups. For a single line $L_2$ with $a_2\notin\mathbb F_p$, Proposition~\ref{p.4 lines case 1}
shows that the two elements
$[M_2],\quad [M_2-I]$
act as translations
\[
t \longmapsto t+u_1,\qquad t\longmapsto t+u_2,
\ \text{
with}\ 
u_1=\frac{b}{a},\qquad u_2=\frac{b}{a-1},
\]
and $u_1,u_2$ $\mathbb F_p$-linearly independent. Hence
$G_{\mathcal L}\cong (C_p)^2$ in this case.

If we add another line $L_3$ we obtain two further translations with parameters
\[
u_3=\frac{b}{a_3},\qquad u_4=\frac{b}{a_3-1}.
\]
For example, taking $a_3=a_2-1$ gives $u_3=u_2$ and $u_4=\dfrac{b}{a_2-2}$, so $U=\Span_{\mathbb F_p}\{u_1,u_2,u_4\}.$
And, for a general choice of $a_2$ , the three
parameters $u_1,u_2,u_4$ are $\mathbb F_p$-linearly independent, and we obtain $G_{\mathcal L}\cong (C_p)^3.$

By adding more lines corresponding to matrices $M_i$ with suitably chosen
$a_i$ one can similarly obtain any elementary abelian $p$-group $(C_p)^m$,
realized as the translation group associated to an $\mathbb F_p$-subspace
$U=\Span_{\mathbb F_p}\{u_1,\dots,u_m\}\subset K$.
\end{remark}

\subsection{The case of the standard construction}\label{s. standard construction} 
In \cite[Chapter~4]{politus1} a class of $(a,b)$-geproci sets, called the
{\it standard construction}, is introduced; see also
\cite[Section~3.3]{politus3} and \cite{ganger2024spreads} for related
discussions of the associated groupoid groups. 

The next proposition rewrites the lines of the standard construction in the
matrix language of this paper. The relevant matrices form a cyclic subgroup of
$\GL_2(K)$, and the associated group $G_{\mathcal L}$ is cyclic.

\begin{proposition}\label{p. standard construction}
    Let $n\ge2$ and let $\varepsilon$ be a $n$-th primitive root of 1. 
Consider the matrix
$$M(\varepsilon)=\left(\begin{array}{cc}
   \varepsilon  & 0 \\
   0 & \varepsilon^{-1}
\end{array}\right).$$

Set $M_{j+1}=M(\varepsilon)^j$
and  $\widetilde{ C_n}=\{M_{j+1}\ |\  j=0,\ldots,n-1\}.$

Let $\mathcal L$ be the set of the $n+2$ lines corresponding to the matrices
$\{M_{0}, M_{\infty}\}\cup \widetilde{ C_n}$ 
 where $M_0=0$ and $M_1=I$ as in Section~\ref{s.lines<->matrices}.
Set $m=\lcm(2,n)$. Then
\[
G_{\mathcal L}
=
\left\{\,\left[
\begin{pmatrix}
(-\varepsilon)^j & 0\\
0 & 1
\end{pmatrix}\right]\ :\ j\in\mathbb Z/m\mathbb Z \right\}
\cong C_{m},
\]
the cyclic group of order $m$ acting by $t\mapsto (-\varepsilon)^j t$ on $\mathbb P^1(K)$.
\end{proposition}
\begin{proof}
By Corollary~\ref{c.describe G_L}, the group $G_{\mathcal L}$ is generated by
the classes $[M_i-M_j]$. Since all matrices $M_i$ are diagonal, all these
classes are diagonal. Hence $G_{\mathcal L}$ is cyclic, generated by the
ratios of the two diagonal entries.

For $i\neq j$, we have
\[
M_i-M_j=
\begin{pmatrix}
\varepsilon^i-\varepsilon^j&0\\
0&\varepsilon^{-i}-\varepsilon^{-j}
\end{pmatrix}.
\]
Thus
\[
[M_i-M_j]
=
\left[
\begin{pmatrix}
1&0\\
0&
\dfrac{\varepsilon^{-i}-\varepsilon^{-j}}
{\varepsilon^i-\varepsilon^j}
\end{pmatrix}
\right]
=
\left[
\begin{pmatrix}
1&0\\
0&-\varepsilon^{-i-j}
\end{pmatrix}
\right].
\]
Therefore $G_{\mathcal L}$ is generated by
$\left[
\begin{pmatrix}
1&0\\
0&-\varepsilon
\end{pmatrix}
\right],$
hence it is cyclic of order
$m=\operatorname{lcm}(2,n).$ 
Equivalently, it has order $2n$ if $n$ is odd and order $n$ if $n$ is even.
\end{proof}
\begin{remark}
The description in Proposition \ref{p. standard construction} makes transparent the asymmetry between the odd and even cases in
the standard construction. In particular, when $n$ is odd and one takes the $n$ grid lines on the quadric
$xy-zw=0$
together with both extra lines $L_0$ and $L_{\infty}$, the associated group
has order $2n$, whereas adding only one of the two extra lines gives an
associated group of order $n$; see \cite[\S 3.3]{politus3} and \cite{ganger2024spreads}. 
\end{remark}

The description in the proposition above also
provides a convenient way to realize the construction over $\mathbb R$.
\begin{remark}
    The set $\widetilde{C}_n$ is a cyclic group in $\GL_2(K)$. The examples from the standard construction can be also obtained with a different representation of this cyclic group.
Namely, let $\vartheta=2\pi/n$ and consider the rotation matrix in $\mathbb R^2$
$$M_{\vartheta}= \left(\begin{array}{rr}
   \cos \vartheta  & -\sin \vartheta \\
   \sin\vartheta & \cos \vartheta
\end{array}\right).$$
\end{remark}

\begin{corollary}Let $\varepsilon$ be a primitive third root of unity and  $s$ be  any $n$-th root of unity with $n\neq 3$. Set 
$$t=\dfrac{\varepsilon(1+s)+s}{1-s}.$$
Let $\mathcal L$ be the set of lines corresponding to the matrices 
$\{M_0,M_{\infty} \} \cup  \widetilde{ C_3} \cup t \widetilde{ C_3}$.
Then $G_{\mathcal L}$ is cyclic of order $\lcm(3,{\rm ord}(s))$.
\end{corollary}
\begin{proof}
First note that 
$$t=\dfrac{\varepsilon(1+s)+s}{1-s}\implies t-\varepsilon=ts+(\varepsilon +1)s\implies s=\dfrac{t-\varepsilon}{t-\varepsilon^2}.$$
From
$$\dfrac{t\varepsilon^j- \varepsilon^i}{t\varepsilon^{-j}- \varepsilon^{-i}}=\varepsilon^{2j}\dfrac{t- \varepsilon^{i-j}}{t- \varepsilon^{j-i}} $$
 we only need to show that 
$\dfrac{t- \varepsilon^{i-j}}{t- \varepsilon^{j-i}}$ is a root of unity.
\begin{itemize}
    \item If $i-j=0$ mod 3 then we get $\dfrac{t- 1}{t- 1}=1$ and we are done  
    \item If $i-j=1$ mod 3 then we get $\dfrac{t- \varepsilon}{t- \varepsilon^2}=s,$ that is by assumption a root of unity.
    \item If $i-j=2$ mod 3 then we get $\dfrac{t- \varepsilon^2}{t- \varepsilon}=\dfrac{1}{s},$ which is again a root of unity.
\end{itemize}
Then $G_{\mathcal L}$ is finite and is generated by the root of unity $\varepsilon\cdot s.$
\end{proof}

\begin{remark}
    For $s=-1$ we get $t=-\dfrac{1}{2}$ which gives, up to scaling, \cite[Example 3.2.10(2)]{politus3}.
\end{remark}

\begin{question}
For which $n>3$ we can find $t$ such that the set of lines $\mathcal L$  corresponding to the matrices in   $\{M_0,M_{\infty} \} \cup  \widetilde{ C_n} \cup t \widetilde{ C_n}$ is skew and the group $G_{\mathcal L}$ has finite order?
\end{question}
From \cite[Example 3.2.10(1)]{politus3} we know that there is at least one example when $n=5$ and $t=-\varphi^2$. Indeed the group $G_{\mathcal L}$ corresponding to 
   $\{M_0,M_{\infty} \} \cup  \widetilde{C_5} \cup -\varphi^2 \widetilde{C_5}$
has order~$10.$
We can make weaker the above question by asking the following. 
\begin{question}
    For which $n>3$ we can find $M$ such that the set of lines $\mathcal L$  corresponding to the matrices in   $\{M_0,M_{\infty} \} \cup \widetilde{C_n} \cup   \{M\} $ is skew and  the group $G_{\mathcal L}$ has finite order?
\end{question}

Another natural question is the following realization problem.
\begin{problem}
Fix a finite subgroup $G\subset \PGL_2(K)$. Characterize the configurations of
skew lines $\mathcal L$ in $\PP^3_K$ such that
$G_{\mathcal L}\cong G.$
\end{problem}

We now illustrate the realization problem in the first nontrivial cyclic case,
showing that the case $G_{\mathcal L}\cong C_3$ already exhibits interesting differences between characteristic $2$ and the other characteristics.

\begin{proposition}\label{p.3 elements}
Let
$M_2=\begin{pmatrix} a & b \\ c & d \end{pmatrix},$
and let $\mathcal L$ be the set of skew lines corresponding to
$\{M_0,M_\infty,M_1,M_2\}.$ 
Then
\[
G_{\mathcal L}\cong C_3
\quad\text{if and only if}\quad
\operatorname{tr}(M_2)=\det(M_2)=1.
\]
\end{proposition}

\begin{proof}
By Corollary~\ref{c.describe G_L} , the group $G_{\mathcal L}$ is generated by $[M_2]$, and $[M_2-I].$ Assume first that $\operatorname{tr}(M_2)=\det(M_2)=1$. By Cayley--Hamilton,
\[
M_2^2-\operatorname{tr}(M_2)M_2+\det(M_2)I=0
\qquad \text{becomes}
\qquad
M_2^2-M_2+I=0.
\]
Hence
$M_2-I=-M_2^2,$
so in $\PGL_2(K)$ we have
$[M_2-I]=[M_2^2]=[M_2]^{-1}.$
Moreover,
$M_2^3=-I,$
therefore $[M_2]$ has order $3$. It follows that
\[
G_{\mathcal L}=\langle [M_2],[M_2-I]\rangle=\langle [M_2]\rangle\cong C_3.
\]

Conversely, assume that $G_{\mathcal L}\cong C_3$. Then $[M_2]\neq [I]$, so
$[M_2]$ has order $3$. Also $[M_2-I]\neq [M_2]$, since otherwise $M_2-I$
would be proportional to $M_2$, forcing $M_2$ to be scalar and $G_{\mathcal L}$
to be trivial. Hence 
$[M_2-I]=[M_2]^2,$ 
so there exists $\lambda\in K^\ast$ such that 
$M_2-I=\lambda M_2^2.$
Equivalently,
$\lambda M_2^2-M_2+I=0.$ 
Since $M_2$ is not scalar, its minimal polynomial has degree $2$. The relation 
$
\lambda M_2^2-M_2+I=0
$ 
shows that the minimal polynomial of $M_2$ divides
$\lambda x^2-x+1,$
hence it is equal to the monic polynomial
$x^2-\frac{1}{\lambda} x+\frac{1}{\lambda}.$

Therefore the characteristic polynomial of $M_2$ is
$x^2-\frac{1}{\lambda} x+\frac{1}{\lambda},$
and in particular
\[
\operatorname{tr}(M_2)=\det(M_2)=\frac{1}{\lambda}.
\]

Now $[M_2]$ has order $3$, so $M_2^3$ is scalar. Using
\[
M_2^2=\frac{1}{\lambda} M_2-\frac{1}{\lambda} I
\qquad
\text{we compute}
\qquad
M_2^3=\left(\frac{1}{\lambda^2}-\frac{1}{\lambda}\right)M_2-\frac{1}{\lambda^2}I.
\]
Since $M_2$ is not scalar, the coefficient of $M_2$ must vanish, hence
$\frac{1}{\lambda^2}-\frac{1}{\lambda}=0,$
so $\lambda=1$. Therefore 
$\operatorname{tr}(M_2)=\det(M_2)=1.$
\end{proof}
\begin{remark}
The proof of Proposition \ref{p.3 elements} is valid in every characteristic.
In characteristic $3$, the polynomial $x^2-x+1$ becomes $(x+1)^2$, so the
corresponding matrix is not semisimple; nevertheless its projective class still
has order $3$.
\end{remark}
\begin{remark}\label{r.3 elements}
Assume that the equivalent conditions of Proposition \ref{p.3 elements} hold, so
that $G_{\mathcal L}\cong C_3$. Then $M_2$ satisfies
$(M_2)^2-M_2+I=0.$

Suppose that $\mathcal L$ contains a fifth line, corresponding to another matrix
$M_3$. Since $[M_3]$ and $[M_3-I]$ both belong to
\[
G_{\mathcal L}=\{[I],[M_2],[(M_2)^2]\},
\]
a straightforward verification shows that necessarily either $M_3=M_2$ or
$M_3=I-M_2$. Thus the only
possible additional matrix is
$M_3=I-M_2.$

If $\operatorname{char}(K)\neq 2$, then
$M_2-M_3=2M_2-I$
is not projectively equivalent to any element of 
$C_3=\{[I],[M_2],[(M_2)^2]\},$ 
so this is impossible. Hence, in characteristic different from $2$, a
configuration with $G_{\mathcal L}\cong C_3$ can contain at most four lines.

In characteristic $2$, however, one has
\[
I-M_2=(M_2)^2
\qquad\text{and}\qquad
M_2-(M_2)^2=I,
\]
so a configuration with five lines and group $G_{\mathcal L}\cong C_3$ may
occur. Precisely, in this case the five associated matrices are
\[
M_0,\qquad M_\infty,\qquad M_1=I,\qquad M_2,\qquad M_3=I-M_2.
\]

In particular, no further matrix can be added while keeping
$G_{\mathcal L}\cong C_3$. Thus a configuration with group $C_3$ contains at
most four lines if $\operatorname{char}(K)\neq 2$, and at most five lines if
$\operatorname{char}(K)=2$.
\end{remark}
\begin{example}
In characteristic $2$, an explicit example of a configuration with five lines
and group $G_{\mathcal L}\cong C_3$ is obtained by taking
\[
M_2=
\begin{pmatrix}
0 & 1\\
1 & 1
\end{pmatrix}.
\]
Then 
$\operatorname{tr}(M_2)=\det(M_2)=1,$
so Proposition \ref{p.3 elements} gives
$G_{\mathcal L}\cong C_3.$ 
Moreover,
\[
M_2^2=
\begin{pmatrix}
1 & 1\\
1 & 0
\end{pmatrix}
=I-M_2,
\]
and therefore the five associated matrices are
\[
M_0=0,\qquad M_\infty,\qquad M_1=I,\qquad
M_2=
\begin{pmatrix}
0 & 1\\
1 & 1
\end{pmatrix},
\qquad
M_3=
\begin{pmatrix}
1 & 1\\
1 & 0
\end{pmatrix}.
\]
Since
\[
M_2-I=M_3,\qquad M_3-I=M_2,\qquad M_2-M_3=I,
\]
all relevant differences are invertible, and hence these matrices define a
configuration of five skew lines with $G_{\mathcal L}\cong C_3$. For a suitable choice of $p$, the corresponding groupoid orbit yields a $(3,5)$-geproci set supported on five skew lines, as in  \cite[Theorem 1]{kettinger2023}. 
\end{example}

\begin{remark}
Assume that the equivalent conditions of Proposition \ref{p.3 elements} hold,
that $\operatorname{char}(K)\neq 2$, and hence that $\mathcal L$ consists of
exactly four lines. Then, for any point $p$ on one of the four lines of $\mathcal L$ and outside the
two transversals, the stabilizer of $p$ in $G_{\mathcal L}$ is trivial. Hence the
$G_{\mathcal L}$-orbit of $p$ on each line has cardinality~$3$.

Since the groupoid $C_{\mathcal L}$ identifies the four lines of the configuration,
the full groupoid orbit $[p]_{\mathcal L}$ consists of $3$ points on each of the
four lines. In particular, $[p]_{\mathcal L}$ is a $[3,4]$-half grid geproci set known as the configuration $D_4$. The fact that, in characteristic 0, the configuration $D_4$ is the only nontrivial $(3,b)$-geproci set was proved in \cite[Theorem 4.10]{politus1}.
\end{remark}

\section{Non-abelian case}\label{s.non-abelian}
In this section we study the non-abelian case.  We first prove that no dihedral group $D_n$ with $n\geq 3$ can occur as $G_{\mathcal L}$ in non-modular characteristic, and then construct explicit configurations realizing the remaining polyhedral groups.

\begin{remark}
In order for $G_{\mathcal L}$ to be non-abelian,
the configuration $\mathcal L$ must contain at least five lines. Indeed,
we have already seen that for four lines 
$\{L_0,L_\infty,L_1,L_2\}$ the group $G_{\mathcal L}$ is always abelian.
Thus any non-abelian example requires at least one further line $L_3$,
whose associated transformations do not commute with those coming from
$L_2$.
\end{remark}

\begin{remark}\label{rem:finite-order-PGL2}
Let $K$ be an algebraically closed field of characteristic $p\ge 0$, and let
$g\in \PGL_2(K)$ be an element of finite order $n$.

Over an algebraically closed field, every finite-order element of
$\PGL_2(K)$ is, up to conjugation, either semisimple or unipotent. In
characteristic $p>0$, a nontrivial unipotent element is represented by a
matrix $U=I+N$, with $N^2=0$, and satisfies
\[
U^k=(I+N)^k=I+kN.
\]
Hence $U^k\neq I$ for $1\le k<p$, while $U^p=I$. Thus every nontrivial
unipotent element of $\PGL_2(K)$ has order exactly $p$.

It follows that:
\begin{itemize}
\item If $p>0$ and $p\mid n$, then necessarily $n=p$, and $g$ is
represented by a nontrivial unipotent element.

\item If $p=0$, or if $p>0$ and $p\nmid n$, then $g$ is semisimple. More
precisely, after a change of basis, we may assume
\[
g=
\left[
\begin{pmatrix}
\lambda & 0\\
0       & 1
\end{pmatrix}
\right],
\]
where $\lambda\in K^*$ is a primitive $n$-th root of unity.
\end{itemize}

The second case is the one used below for the order-$n$ cyclic subgroup of a
dihedral group $D_n$.
\end{remark}

\begin{theorem}\label{t.no Dn}
Let $n\ge 3$, and assume that $\operatorname{char}(K)\nmid n$.
    No dihedral group $D_n$ can occur as $G_{\mathcal L}.$ 
\end{theorem}
\begin{proof}
Suppose, by contradiction, that $G_{\mathcal L}\cong D_n$ for some
$n\ge3$ with $\operatorname{char}(K)\nmid n$. Recall
\[
D_n = \langle r,s \mid r^n = s^2 = 1,\ srs^{-1}=r^{-1} \rangle.
\]
Let $r,s\in G_{\mathcal L}$ be elements corresponding to these standard
generators, and choose representatives
\[
r = [A],\qquad s = [B],\qquad A,B\in\GL_2(K).
\]
From Remark~\ref{rem:finite-order-PGL2}, since $r$ has order $n\ge3$ and $\operatorname{char}(K)\nmid n$, after a change of basis,  we may assume
\[
  r = \left[
  \begin{pmatrix}
    \lambda & 0\\
    0       & 1
  \end{pmatrix}
  \right],
\]
with $\lambda\in K^*$ a primitive $n$-th root of unity.

We apply the same change of basis simultaneously to all matrices $M_i$, and we keep the notation $M_i$ for the transformed matrices. Since $M_0=0$ and $M_1=I$, the normalization $M_0=0$, $M_1=I$ is preserved. 
Moreover the matrices representing the maps $F_{ijk}$, and hence the closed
compositions generating $G_{\mathcal L}$, are merely conjugated in $\PGL_2$.

In this basis the eigenspaces of $r$ are $E_1=K\cdot(1,0)$ and $E_2=K\cdot(0,1)$.
The relation $[BAB^{-1}]=[A^{-1}]$ implies that $s$ permutes $\{E_1,E_2\}$, so
$B$ is (up to scalar) either diagonal or anti-diagonal in this basis. Hence
every element of $\langle r,s\rangle\cong D_n$ is represented in this basis
by a diagonal or an anti-diagonal matrix. There are two cases:
\begin{itemize}
  \item If $B$ fixes both $E_1$ and $E_2$, then in this basis $B$ is diagonal,
  and therefore commutes with $A$. This would give $srs^{-1}=r$, contradicting
  $srs^{-1}=r^{-1}\neq r$ for $n\ge3$.

  \item Therefore $B$ must swap $E_1$ and $E_2$, and in this same basis we may
  write
  \[
  B =
  \begin{pmatrix}
  0 & b\\
  c & 0
  \end{pmatrix},
  \qquad b,c\in K^*.
  \]
\end{itemize}
Thus, in a suitable basis of $K^2$, the rotation generator $r$ of $D_n$ is
represented by a diagonal matrix and the reflection generator $s$ by an
anti-diagonal matrix. It follows that every element of the subgroup
$\langle r,s\rangle\cong D_n$ is represented, in this basis, by either a diagonal or an anti-diagonal matrix: powers $r^k$ are diagonal, and elements $sr^k$ are anti-diagonal.

Since $G_{\mathcal L}=\langle r,s\rangle$, all generators must also be represented, in the chosen basis, by diagonal or anti-diagonal matrices.
By Corollary~\ref{c.describe G_L}, for each $i$ the classes
$[M_i]=[M_i-M_0]$ and $[I-M_i]=[M_1-M_i]$ belong to $G_{\mathcal L}$; therefore $M_i$ and $I-M_i$ must be represented by either diagonal or anti-diagonal matrices.
This forces $M_i$ to be diagonal. Indeed, if $M_i$ were anti-diagonal, then, since $M_i$ is invertible,
$I-M_i$ would have both nonzero diagonal and nonzero off-diagonal entries,
and hence would be neither diagonal nor anti-diagonal. Thus $M_i$ is diagonal.

Thus, since all $M_i$ are diagonal in the same basis, all differences $M_i-M_j$ are diagonal. By Corollary~\ref{c.describe G_L}, the group $G_{\mathcal L}$ is generated by the classes $[M_i-M_j]$. Hence all generators commute, so $G_{\mathcal L}$ is abelian. This contradicts our assumption that $G_{\mathcal L}\cong D_n$ is non-abelian.
\end{proof}
\begin{remark}
The hypothesis $\operatorname{char}(K)\nmid n$ is necessary for the argument
above. Indeed, if $\operatorname{char}(K)=p>0$ and $n=p$, then the cyclic
subgroup $C_p$ of $D_p$ can be represented by a unipotent element rather
than by a semisimple one. For example, the transformations
\[
r:t\mapsto t+1,
\qquad
s:t\mapsto -t
\]
satisfy $r^p=s^2=1$ and $srs^{-1}=r^{-1}$, and hence generate a subgroup of
$\PGL_2(K)$ isomorphic to $D_p$. In particular,
Example~\ref{non-abelian affine} realizes $D_3\cong S_3$ when $p=3$.

On the other hand, if $p\mid n$ and $n\neq p$, then $D_n$ cannot embed in
$\PGL_2(K)$, since $\PGL_2(K)$ has no element of order $n$ in this case.
\end{remark}

We now turn to explicit realizations in the non-abelian case. The examples in this section show that the realization problem is genuinely nontrivial beyond the abelian setting: while Theorem~\ref{t.no Dn} excludes non-modular dihedral groups, the three polyhedral groups $A_4$, $S_4$, and $A_5$ do occur as groups $G_{\mathcal L}$ of suitable configurations of skew lines. Thus the matrix description developed in Section \ref{s.lines<->matrices} is strong enough not only to rule out certain finite groups, but also to produce the classical non-abelian finite subgroups of $\PGL_2(K)$ in a concrete geometric way.

The following examples are obtained by starting from the classical
representations of the polyhedral groups in $\PGL_2(\mathbb C)$.  We choose
matrices $M_2$ and $M_3$ so that some of the classes
\[
[M_2],\quad [M_3],\quad [I-M_2],\quad [I-M_3],\quad [M_2-M_3]
\]
satisfy the standard presentations of $A_4$, $S_4$, and $A_5$.  The
invertibility of all differences $M_i-M_j$ guarantees that the corresponding
lines are pairwise skew.  The computations below then verify that the group
generated by all the difference classes is exactly the desired finite subgroup
of $\PGL_2(\mathbb C)$.

\begin{example}[$G_{\mathcal L}\cong A_5$]\label{ex.A5}
Work over $K=\mathbb C$ and set $i^2=-1$ and $\varphi=\frac{1+\sqrt{5}}{2}$.
Consider the matrices
\[
   M_2=\frac{1}{2} 
   \begin{pmatrix}
     1+i & 1+i\\
     -1+i & 1-i
   \end{pmatrix},
   \qquad
   M_3=\frac{1}{2} 
   \begin{pmatrix}
     \varphi+\varphi^{-1}i & 1\\
     -1 & \varphi-\varphi^{-1}i
   \end{pmatrix}.
\]
Let $\mathcal L$ be the set of lines corresponding to the matrices 
$\{M_{\infty},M_0,M_1,M_2,M_3\}.$ 
Let $r=[M_2]$ and $s=[M_3]\in \PGL_2(\mathbb C)$. 
Since $\det([M_2,M_3])\neq 0$ then $G_{\mathcal L}$ is  not abelian.

A direct computation, using Macaulay2, shows that
\[
  r^3 = s^5 = (rs)^2 = 1
\]
and the subgroup 
$|\langle r,s\rangle| = 60$.

By looking at the action on $\PP^1$ (for
instance via the size of a general orbit on $L_{\infty}$) we find that
$|G_{\mathcal L}| = |H| = 60.$ 
Thus
$G_{\mathcal L} \;\cong\; \langle r,s\rangle \;\cong\; A_5.$ 

For the associated configuration $\mathcal L$, the full groupoid orbit $[p]_{\mathcal L}$
across all five lines has cardinality
\[
|[p]_{\mathcal L}| =
\begin{cases}
  30 & \text{if } p=[0:0:0:1],\\[4pt]
  300 & \text{if } p \text{ is a general point on } L_{\infty}.
\end{cases}
\]
Thus, besides the general orbit on $L_\infty\cong \PP^1$, the action of
$G_{\mathcal L}$ has smaller distinguished orbits corresponding to points with
nontrivial stabilizer. For any such point $p\in L_\infty$, the associated
groupoid orbit $[p]_{\mathcal L}$ yields a half-grid geproci set.
\end{example}

\begin{example}[$G_{\mathcal L}\cong S_4$]\label{ex.S4}
Work over $K=\mathbb C$ and set $i^2=-1$. Consider the matrices
\[
M_2=\frac{1}{2}\begin{pmatrix}
    1+i & 1+i\\
    -1+i & 1-i
\end{pmatrix},
\qquad
M_3=\begin{pmatrix}
    0 & 1\\
    -1 & 0
\end{pmatrix}.
\]
Let $\mathcal L$ be the set of lines corresponding to the matrices 
$\{M_{\infty},M_0,M_1,M_2,M_3\}.$

A direct computation shows that
$\det\bigl([M_2,M_3]\bigr)=\det(M_2M_3-M_3M_2)=2,$
so by Lemma~\ref{l.transv 5 lines} these five lines have no common transversal
in characteristic $\neq 2$.

Let $r=[M_2-I]=\left[\frac{1}{2}\begin{pmatrix}
    -1+i & 1+i\\
    -1+i & -1-i
\end{pmatrix}\right],\, s=[M_2-I]\cdot [M_3-I]=\left[\begin{pmatrix}
    i & 1\\
    -1 & -i
\end{pmatrix}\right]\in\PGL_2(\mathbb C)$. One checks, using Macaulay2, that
the subgroup $H=\langle r,s\rangle$ has order 24 and 
presentation
\[
 r^3=s^2=(rs)^4=1,
\]
 By inspecting the action on $\PP^1$ (for
instance via the size of a general orbit) we find that 
$|G_{\mathcal L}| = |H| = 24,$ 
and hence $G_{\mathcal L}\cong S_4$.

For the associated configuration $\mathcal L$, the full groupoid orbit $[p]_{\mathcal L}$
across all five lines has cardinality
\[|[p]_{\mathcal L}|=\begin{cases}
120& \text{for a general point}\ p\ \text{on}\  L_{\infty};\\
40 & \text{for}\ p=[0:0:-1:(1-i)(1+\sqrt{3})/2];\\
30 & \text{for}\ p=[0:0:0:1].\\
 \end{cases}\]
In positive characteristic additional orbit degeneracies may occur.
\end{example}

\begin{example}[$G_{\mathcal L}\cong A_4$]\label{ex.A4}
Work over  $K=\mathbb C$ and let $\varepsilon$ be a primitive six-th root of unity. Consider the set $\mathcal L$ of the 5 lines corresponding to the matrices: 
    $$\left\{M_{\infty},\,M_0,\,M_1,\,M_2=\left(\!\begin{array}{cc}
\varepsilon&a\\
0&\varepsilon^{-1}
\end{array}\!\right),\,M_3=\left(\!\begin{array}{cc}
\varepsilon&0\\
a^{-1}&\varepsilon^{-1}
\end{array}\!\right)\right\},\ \ \ a\in K^*.$$

A direct computation shows that 
$
\det\bigl([M_2,M_3]\bigr)=\det(M_2M_3-M_3M_2)=2,
$ 
so by Lemma~\ref{l.transv 5 lines} these five lines have no common transversal
in characteristic $\neq 2$.

Let $r=[M_2],\, s=[M_2 M_3]=\left[\begin{pmatrix}
    \varepsilon & a\varepsilon^{-1}\\
    (a\varepsilon)^{-1}& \varepsilon^{-2}
\end{pmatrix}\right]\in\PGL_2(\mathbb C)$. 

A direct computation, using Macaulay2, shows that  the group $H=\langle r,s\rangle$ has order $12$ and  $s^2=r^3=(sr)^3=1$. Hence $H\cong A_4$. 
 By looking at the action on $\PP^1$ (for
instance via the size of a general orbit on $L_{\infty}$) we find that 
$|G_{\mathcal L}| = |H| = 12,$ 
and hence $G_{\mathcal L}\cong A_4$.

For the associated configuration $\mathcal L$, the full groupoid orbit $[p]_{\mathcal L}$
across all five lines has cardinality
\[|[p]_{\mathcal L}|=\begin{cases}
60& \text{for a general point}\ p\ \text{on}\  L_{\infty};\\
30 & \text{for}\ p=[0:0:-1:ia];\\
20 & \text{for}\ p=[0:0:0:1].\\
 \end{cases}\]
The orbit of $[0:0:0:1]$ on $L_{\infty}$, $[p]_{\mathcal L}\cap L_{\infty}$,  corresponds to the set of 4 equianharmonic points: $ \left(\infty,\,\frac{1}{\varepsilon a},\,0 ,\,-\frac{\varepsilon }{a}\right).$
\end{example}

\begin{example}[A non-abelian affine $p$-semi-elementary group]\label{non-abelian affine}
Let $K$ be an algebraically closed field of characteristic $p>2$, and
identify $\mathbb F_p\subset K$.  
Choose $a\in\mathbb F_p^*$ such that $a^2$ has order $p-1$ in $\mathbb F_p^*$, and set
\[
M_2=
\begin{pmatrix}
-1 & 1 \\
0  & -1
\end{pmatrix},
\qquad
M_3=
\begin{pmatrix}
a & 0 \\
0 & a^{-1}
\end{pmatrix}.
\]
Let $\mathcal L$ be the configuration 
$\mathcal L=\{L_\infty,L_0,L_1,L_2,L_3\}$ 
corresponding to the matrices $M_\infty,M_0,M_1,M_2,M_3$ as in Section~2.
 Since
$[M_2M_3]\neq [M_3M_2],$
the group $G_{\mathcal L}$ is non-abelian.

On the projective line $\PP^1_K$ with affine coordinate $t=x/y$, the
Möbius transformations induced by $[M_2]$ and $[M_3]$ are a nontrivial translation of order $p$, and a dilation by the scalar $a^2\in\mathbb F_p^*$ of order $p-1$, respectively:

\[
[M_2] : t \longmapsto \frac{-t+1}{-1} = t-1; \qquad
[M_3] : t \longmapsto \frac{at}{a^{-1}} = a^2 t.
\]
Thus
$\langle [M_2]\rangle \cong C_p,$ $
\langle [M_3]\rangle \cong C_{p-1},$
and the group 
$H = \langle [M_2],[M_3]\rangle$
is the affine group generated by a translation and a dilation:
\[
H \;\cong\; C_p \rtimes C_{p-1},
\]
of order $p(p-1)$.

Now note that all matrices $M_i$ in this example are upper triangular, and so
are all differences $M_i-M_j$. Hence every generator $F_{ijk} =
(M_j-M_k)^{-1}(M_i-M_k)$ of $G_{\mathcal L}$ has the form
\[
t \longmapsto \alpha t + \beta,\qquad \alpha\in\mathbb F_p^*,\ \beta\in\mathbb F_p,
\]
and therefore lies in the affine subgroup $H$. Thus
$G_{\mathcal L} \subseteq H.$
On the other hand $[M_2],[M_3]\in G_{\mathcal L}$ by Lemma~\ref{l.describe f_ijk},
so $H\subseteq G_{\mathcal L}$ and hence $$G_{\mathcal L}=H\cong C_p\rtimes C_{p-1}.$$
This example shows that the matrix formalism developed in Section
\ref{s.lines<->matrices} captures not only the polyhedral finite subgroups of
$\PGL_2(K)$, but also the affine positive-characteristic part of the
classification. In particular, the fact that all matrices $M_i$ and all
differences $M_i-M_j$ are upper triangular forces every generator $F_{ijk}$ to
lie in the affine subgroup of $\PGL_2(K)$, giving a concrete geometric
realization of the non-abelian group $C_p\rtimes C_{p-1}$.

For a general point on $L_\infty$ the
stabilizer in $G_{\mathcal L}$ is trivial, so
\[
|[p]_{\mathcal L}\cap L_\infty| = |G_{\mathcal L}| = p(p-1).
\]
The orbit of the point $[0:0:0:1]\in L_{\infty}$ under $G_{\mathcal L}$
has size $p$ on $L_{\infty}$: indeed, the translation $t\mapsto t-1$
acts transitively on the $p$ points of the affine line, and dilations
preserve this orbit. On the other hand, the point
$[0:0:1:0]\in L_{\infty}$ corresponds to $t=\infty$ and is fixed by both
$[M_2]$ and $[M_3]$, so its orbit has size $1$.
\end{example}

We conclude with a brief remark connecting our examples to groups of Lie type and formulating some natural open questions.
\begin{remark}
The examples above are consistent with the well-known exceptional isomorphisms
\[
\PSL_2(\mathbb F_3)\cong A_4,\qquad
\PGL_2(\mathbb F_3)\cong S_4,\qquad
\PSL_2(\mathbb F_4)\cong A_5.
\]
In this sense, the groups realized in Examples \ref{ex.A4}, \ref{ex.S4}, and
\ref{ex.A5} may also be viewed as the first instances of groups of Lie type
appearing in the geometry of skew-line configurations.

This suggests a broader realization problem: which finite subgroups of
$\PGL_2(K)$ can occur as $G_{\mathcal L}$ for a configuration of skew lines in
$\PP^3$? In particular, it would be interesting to understand whether further
groups of Lie type, such as $\PSL_2(\mathbb F_q)$ or $\PGL_2(\mathbb F_q)$ for
larger $q$, can arise from suitable configurations.
\end{remark}

\section*{Computational verification}

The following Macaulay2 code was used to compute orbits of points under the geometric maps $f_{ijk}$. The functions implement the construction directly: given a point $P\in L_i$, the plane spanned by $P$ and another line $L_j$ is intersected with a third line to obtain the image point. The final block is a parameterized example with four lines; it can be specialized to the examples in the paper by assigning the corresponding matrices.

\subsection*{Computing the orbit via the geometric construction}\

{\footnotesize \noindent{\it Check if the lines in a set are skew.}
\begin{verbatim}
isSkew = (L1,L2) -> (
    inter = saturate(L1+L2);
    Linter := select(flatten entries gens inter, f -> degree f == {1});
    if #Linter == 0 then return true else  return false
    );

checkLL = (LL) -> (
    for i from 0 to #LL-2 do (
        for j from i+1 to #LL-1 do (
            if isSkew(LL#i, LL#j)==false then (return false))
        ); return true
    );
\end{verbatim}

\noindent{\it Compute the image of a point under the action of all the lines in a set. If the orbit is finite, we have to run \texttt{orbitL} until the set stabilizes.}
\begin{verbatim}
planeSpan = (P,L) -> (
    I := intersect(P,L);
    Lgens := select(flatten entries gens I, f -> degree f == {1});
    ideal mingens ideal Lgens
    );

newPoint = (P,Li,Lj) -> (saturate(planeSpan(P,Lj) + Li));

memberPoint = (P, plist) -> (any(plist, Q -> isSubset(P,Q) and isSubset(Q,P)));

orbit = (P0, LL) -> ( points = {P0};
    for i from 0 to #LL-1 do ( Li = LL#i;
        if ((Li + P0) == ideal(x,y,z,w)) then
        for j from 0 to #LL-1 do ( Lj = LL#j;
            if (i!=j and (Lj + P0) == ideal(x,y,z,w)) then(
                Pnew=newPoint(P0,Li,Lj);
                if not memberPoint(Pnew, points) then points = append(points,Pnew);
                ))); points
    );

orbitL =(LP,LL,START) ->( points = LP;
    for k from START to #LP-1 do(
        P0=LP#k;
        for i from 0 to #LL-1 do ( Li = LL#i;
            if ((Li + P0) == ideal(x,y,z,w)) then
            for j from 0 to #LL-1 do ( Lj = LL#j;
                if (i!=j and (Lj + P0) == ideal(x,y,z,w)) then(
                    Pnew=newPoint(P0,Li,Lj);
                    if not memberPoint(Pnew, points) then points = append(points,Pnew);
                    )))); points
    );
\end{verbatim}

\noindent{\it Line from matrix $x$ is $M_{\infty}$; from $0$ is $M_0$}
\begin{verbatim}
MtoL = (M) -> (if M==x then ideal(x,y)  else ideal(M*matrix{{x},{y}}-matrix{{z},{w}}))
\end{verbatim}
\noindent{\it Computing the orbit of the point P0 over 4 lines and a  general matrix.}
\begin{verbatim}
K = frac(QQ[a,b,c,d]);
R= K[x,y,z,w]

MM={x, 0, matrix{{1,0},{0,1}},  matrix{{a,b},{c,d}}}
LL=toList apply(0..#MM-1, i-> MtoL(MM_i)); #LL
checkLL(LL)

P0=ideal (sort{x,y,z+w})

ids0=orbitL({P0}, LL,0); #ids0
ids1=orbitL(ids0, LL, 1); #ids1
ids2=orbitL(ids1, LL, #ids0); #ids2
ids3=orbitL(ids2, LL, #ids1); #ids3 
\end{verbatim}
}


\begin{thebibliography}{10}
	
	\bibitem{Beauville2010}
	{\sc A.~Beauville}, {\em Finite subgroups of {PGL}$_2(k)$}, Contemporary
	Mathematics, 522 (2010), pp.~23--29.
	\newblock In \emph{Vector bundles and complex geometry}.
	
	\bibitem{politus1}
	{\sc L.~Chiantini, {\L }.~Farnik, G.~Favacchio, B.~Harbourne, J.~Migliore,
		T.~Szemberg, and J.~Szpond}, {\em Configurations of points in projective
		space and their projections}, 2022, arXiv:2209.04820.
	
	\bibitem{politus3}
	{\sc L.~Chiantini, {\L}.~Farnik, G.~Favacchio, B.~Harbourne, J.~Migliore,
		T.~Szemberg, and J.~Szpond}, {\em Combinatorics of skew lines in
		$\mathbb{P}^3$ with an application to algebraic geometry}, 2025,
	arXiv:2308.00761v2.
	
	\bibitem{chiantini2021}
	{\sc L.~Chiantini and J.~Migliore}, {\em Sets of points which project to
		complete intersections, and unexpected cones}, Transactions of the American
	Mathematical Society, 374 (2021), pp.~2581--2607.
	
	\bibitem{DolgachevIskovskikh2009}
	{\sc I.~V. Dolgachev and V.~A. Iskovskikh}, {\em Finite subgroups of the plane
		{C}remona group}, in Algebra, Arithmetic, and Geometry. Vol.~I, Y.~Tschinkel
	and Y.~Zarhin, eds., vol.~269 of Progress in Mathematics, Birkh{\"a}user,
	Boston, 2009, pp.~443--548.
	
	\bibitem{Faber2023}
	{\sc X.~Faber}, {\em Finite p-irregular subgroups of {PGL}$_2$(k)}, La
	Matematica, 2 (2023), pp.~479--522.
	
	\bibitem{FM2024}
	{\sc G.~Favacchio and J.~Migliore}, {\em On the {W}eak {L}efschetz {P}roperty
		for certain ideals generated by powers of linear forms}, {to appear in} Kyoto
	Journal of Mathematics,  (2025).
	\newblock Preprint arXiv:2406.09571.
	
	\bibitem{ganger2024spreads}
	{\sc A.~J. Ganger}, {\em Spreads and Transversals and Their Connection to
		Geproci Sets}, PhD thesis, The University of Nebraska-Lincoln, 2024.
	
	\bibitem{macaulay2}
	{\sc D.~R. Grayson and M.~E. Stillman}, {\em Macaulay2, a software system for
		research in algebraic geometry}, 2002.
	\newblock Available at \url{http://www.math.uiuc.edu/Macaulay2}.
	
	\bibitem{hartshorne1977book}
	{\sc R.~Hartshorne}, {\em Algebraic geometry}, Graduate Texts in Mathematics,
	52 (1977).
	
	\bibitem{kettinger2023}
	{\sc J.~Kettinger}, {\em On the Superabundance of Singular Varieties in
		Positive Characteristic}, PhD thesis, The University of Nebraska-Lincoln,
	2023.
	
	\bibitem{kettinger2024}
	\leavevmode\vrule height 2pt depth -1.6pt width 23pt, {\em The geproci property
		in positive characteristic}, Proceedings of the American Mathematical
	Society, 152 (2024), pp.~3229--3242.
	
	\bibitem{kettinger2025}
	\leavevmode\vrule height 2pt depth -1.6pt width 23pt, {\em Finite groupoids of
		configurations of lines in $\mathbb{P}^3_{\mathbb{C}}$}, arXiv preprint
	arXiv:2511.05454,  (2025).
	
	\bibitem{Klein1884}
	{\sc F.~Klein}, {\em Vorlesungen {\"u}ber das Ikosaeder und die Aufl{\"o}sung
		der Gleichungen vom f{\"u}nften Grade}, Teubner, Leipzig, 1884.
	
	\bibitem{Klein1956}
	\leavevmode\vrule height 2pt depth -1.6pt width 23pt, {\em Lectures on the
		Icosahedron and the Solution of Equations of the Fifth Degree}, Dover
	Publications, New York, 1956.
	\newblock Reprint of the 2nd, revised edition of the German original.
	
	\bibitem{leuschke2012}
	{\sc G.~J. Leuschke and R.~Wiegand}, {\em {C}ohen-{M}acaulay representations},
	vol.~181 of Mathematical Surveys and Monographs, Providence, RI, American
	Mathematical Society, 2012.
	
\end{thebibliography}
\end{document}